\documentclass[12pt]{article}
\usepackage{amsmath}
\usepackage{amsthm}
\usepackage{amssymb}

\theoremstyle{theorem}
\newtheorem{theorem}{Theorem}
\newtheorem{lemma}{Lemma}
\newtheorem{corollary}{Corollary}

\newcommand{\Var}{{\rm Var}}

\begin{document}

\vspace*{30px}

\begin{center}\Large
\textbf{Lah numbers and Lindstr\"{o}m's lemma}
\bigskip\large

Ivica Martinjak\\
Faculty of Science, University of Zagreb\\
Zagreb, Croatia\\
{\tt imartinjak@phy.hr} \\
\bigskip
Riste \v Skrekovski\\
Faculty of Information Studies, Novo Mesto, Slovenia \\
FMF, University of Ljubljana, Ljubljana, Slovenia\\
FAMNIT, University of Primorska, Slovenia \\
{\tt skrekovski@gmail.com}


\end{center}


\begin{abstract} 
We provide a combinatorial interpretation of Lah numbers by means of planar networks. Henceforth, as a conesquence of Lindstr\"{o}m's lemma, we conclude that the related Lah matrix possesses a remarkable property of total non-negativity. 
\end{abstract}


\section{Introduction}

The Lah numbers were introduced by Ivan Lah in 1952 and since then they are the subject of many prominent researches. For $n, k \in \mathbb{N}_0$ we define $L_{n.k}$ as the number of ways to partition the set $[n] = \{1,2, \ldots, n\}$ into $k$ nonempty tuples (i.e. linearly ordered sets). We let $L_{0,0}:=1$. Define the {\em Lah matrix} ${\rm LM}_m=[L_{i,j}]$ as the matrix of dimension $m\times m$, whose element in the $i$-th row and $j$-th column is $L_{i,j} $. Note that ${\rm LM}_m$ is a low-triangular matrix. For the first column of LM it holds $L_{m,1}=m!$ since all labeled "balls" we have to put into sole "box" - where we distinguish the order of balls meaning that we deal with permutations of $n$. Further consideration of these partitioning shows that the Lah numbers are recursive in nature, and more precisely 
$$
L_{n+1,k} = L_{n, k-1} + (n+k)L_{n,k}.
$$
Other contexts where the Lah numbers appears include {noncrossing partitions}, {Dyck paths}, $q$-analogues as well as falling and rising factorials, just to name a few. In particular, the number of Dyck $n$-paths with $n+1-k$ peaks labeled $1,2,\ldots, n+1-k$ is equal to $L_{n,k}$. Lah numbers were originally introduced as coefficients in the polynomial identity
 \begin{equation*} 
     x(x+1) \cdots (x+n-1) = \sum_{k=0}^n \, L_{n,k}\,  x(x-1) \cdots (x-k+1),
\end{equation*}
where $n,k,x \in \mathbb{N}_0$. An explicit formula is known for Lah numbers,
 \begin{equation} \label{form}
     L_{m,k} =    \binom{m-1}{k-1} \frac{m!}{k!} \,.  
\end{equation}
Some natural generalizations are done by Wagner \cite{Wagn} as well as by Ramirez and Shattuck \cite{RaSh}.

\section{The main result}


A matrix is {\em totally non-negative} (resp.~{\em positive}) if each of its minors is non-negative (resp.~positive) \cite{FoZe}. In a planar acyclic weighted directed graph with $n$ sources $a_i$'s and $n$ sinks $b_i$'s, one defines 
a {\em weight matrix} $W=[w_{i,j}]$ of dimension $n\times n$, where $w_{i,j}$ is the sum of the weights of paths from $a_i$ to $b_j$. Such graphs are also called planar networks.

We let $\Delta_{I,J}(M)$ denotes the minor of a matrix $M$ with the row indices from set $I$ and the column indices from set $J$.

\begin{lemma}[Lindstr\"{o}m's lemma]
A minor $\Delta_{I,J}(W)$ of the weight matrix $W$ of a planar network is equal to the sum of weights of all collections of vertex-disjoint paths that connect the sources labeled by $I$ with the sinks labeled by $J$. 
\end{lemma}

We define a planar network $N_n$  by the figure below (Figure \ref{Fig1}). Note that with the same network but with unit weights, we obtain the "Pascal triangle" as the related weight matrix.

\begin{theorem}
For $m, k \leq n$, the Lah number $L_{m,k}$ corresponds to the number of weighted paths in the network $N_n$  from vertex $a_m$ to the vertex $b_k$.  
\end{theorem}
\begin{proof}
This obviously holds for $m<k$ so assume $m \ge k$. Notice that every directed path from $a_m$ to $b_k$ pass through the rectangular grid, which is of size $(m-k)\times (k-1)$  (e.g. for $a_5$ and $b_3$ it is marked in the figure). Thus the number 
of these paths is $$\binom{m-k+k-1}{k-1} = \binom{m-1}{k-1}.$$ Every such path is of length $m-1$ consisting of $k-1$ "horizontal" edges
and $m-k$ "diagonal" edges. Horizontal edges are all of weight 1 and regarding the diagonal edges, when moving from $a_m$ to $b_k$, they have weights $$m, m-1, \ldots, k+1,$$ respectively. So, each such path has weight $\frac{m!}{k!}$. This gives us that total weight of the paths from $a_m$ to $b_k$ is $$\binom{m-1}{k-1} \frac{m!}{k!},$$ which is the Lah number $L_{m,k}$ by (\ref{form}). 
\begin{figure}[h]
\setlength{\unitlength}{0.65cm}

\begin{picture}(0,8.1)(-8.9,-1)

\put(-1,0){\line(1,0){6}}
\put(-1,1){\line(1,0){6}}
\put(-1,2){\line(1,0){6}}
\put(-1,3){\line(1,0){6}}
\put(-1,4){\line(1,0){6}}
\put(-1,5){\line(1,0){6}}
\put(-1,6){\line(1,0){6}}

\put(5,0){\line(-1,1){6}}
\put(5,1){\line(-1,1){5}}
\put(5,2){\line(-1,1){4}}
\put(5,3){\line(-1,1){3}}
\put(5,4){\line(-1,1){2}}
\put(5,5){\line(-1,1){1}}

\put(-2.1, 0){ $a_1$ }
\put(-2.1, 1){ $a_2$ }
\put(-2.1, 2){ $a_3$ }
\put(-2.1, 3){ $a_4$ }
\put(-2.1, 4){ $a_5$ }
\put(-2.1, 6){ $a_n$ }
\put(-2.1, 5.3){ $.$ }
\put(-2.1, 5,0){ $.$ }
\put(-2.1, 4,7){ $.$ }

\put(5.2, 0){ $b_1$ }
\put(5.2, 1){ $b_2$ }
\put(5.2, 2){ $b_3$ }
\put(5.2, 3){ $b_4$ }
\put(5.2, 4){ $b_5$ }
\put(5.2, 6){ $b_n$ }
\put(5.2, 5,3){ $.$ }
\put(5.2, 5,0){ $.$ }
\put(5.2, 4,7){ $.$ }

\put(1,4.04){\line(1,0){2}}
\put(1,4.02){\line(1,0){2}}

\put(3,2.04){\line(1,0){2}}
\put(3,2.02){\line(1,0){2}}

\put(1,4.02){\line(1,-1){2}}
\put(1,4.04){\line(1,-1){2}}

\put(3,4.02){\line(1,-1){2}}
\put(3,4.04){\line(1,-1){2}}

\put(-1,0){\circle*{0.15}}
\put(5,0){\circle*{0.15}}

\put(-1,1){\circle*{0.15}}
\put(4,1){\circle*{0.15}}
\put(5,1){\circle*{0.15}}

\put(-1,2){\circle*{0.15}}
\put(3,2){\circle*{0.15}}
\put(4,2){\circle*{0.15}}
\put(5,2){\circle*{0.15}}

\put(-1,3){\circle*{0.15}}
\put(2,3){\circle*{0.15}}
\put(3,3){\circle*{0.15}}
\put(4,3){\circle*{0.15}}
\put(5,3){\circle*{0.15}}

\put(-1,4){\circle*{0.15}}
\put(1,4){\circle*{0.15}}
\put(2,4){\circle*{0.15}}
\put(3,4){\circle*{0.15}}
\put(4,4){\circle*{0.15}}
\put(5,4){\circle*{0.15}}

\put(0,5){\circle*{0.15}}
\put(1,5){\circle*{0.15}}
\put(2,5){\circle*{0.15}}
\put(3,5){\circle*{0.15}}
\put(4,5){\circle*{0.15}}
\put(5,5){\circle*{0.15}}

\put(-1,5){\circle*{0.15}}
\put(-1,6){\circle*{0.15}}
\put(-0,6){\circle*{0.15}}
\put(1,6){\circle*{0.15}}
\put(2,6){\circle*{0.15}}
\put(3,6){\circle*{0.15}}
\put(4,6){\circle*{0.15}}
\put(5,6){\circle*{0.15}}

\put(-0.4, 5.3){{ \scriptsize $n$ }}
\put(0.6, 5.3){{ \scriptsize $n$ }}
\put(1.6, 5.3){{ \scriptsize $n$ }}
\put(2.6, 5.3){{ \scriptsize $n$ }}
\put(3.6, 5.3){{ \scriptsize $n$ }}
\put(4.6, 5.3){{ \scriptsize $n$ }}

\put(1.6, 3.3){{ \scriptsize $5$ }}
\put(2.6, 3.3){{ \scriptsize $5$ }}
\put(3.6, 3.3){{ \scriptsize $5$ }}
\put(4.6, 3.3){{ \scriptsize $5$ }}

\put(2.6, 2.3){{ \scriptsize $4$ }}
\put(3.6, 2.3){{ \scriptsize $4$ }}
\put(4.6, 2.3){{ \scriptsize $4$ }}

\put(3.6, 1.3){{ \scriptsize $3$ }}
\put(4.6, 1.3){{ \scriptsize $3$ }}

\put(4.6, 0.3){{ \scriptsize $2$ }}

\end{picture}
\caption{The Lah numbers count weighted paths in a directed acyclic planar graph.} \label{Fig1}
\end{figure}
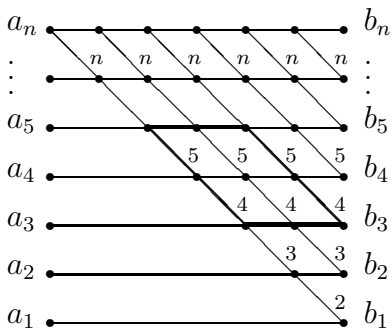
\end{proof}

As an easy consequence from Lindstr\"{o}m's lemma we obtain the following.
\begin{corollary}
The Lah triangular matrix  ${\rm LM}_m$ is totally non-negative.
\end{corollary}

Totally positive matrices and in particular its eigenvalues are related with the {\em variation-decreasing} vectors. Let $u=(u_1, u_2, \ldots, u_n)$ be a vector in $\mathbb{R}^n$. A {\it sign change} in $u$ is a pair of indices ($i$,$j$) such that for $i < j \leq n$: 
\begin{description}
\item $i)$ $u_k=0$ for all $k$ (if there are any), $i<k < j$, and  
\item $ii)$  $u_iu_j < 0$.
\end{description}
The {\em weak variation} $\Var^{-}(u)$ is the number of signs changes in $u$. For example, $\Var^{-} (2,-2,0,1,-3,0,0,1)=4$. Now, an $n \times m$ matrix $M$ with real entries is {\em variation-decreasing} if for all nonzero vectors $x \in \mathbb{R}^m$
\begin{equation}  \label{vardecr}
\Var^{-}(Mx) \leq \Var^{-}(x).
\end{equation} 
We point out Motzkin's theorem that relates the notion of variation-decreasing matrices with total positivity (see J. Kung, G. Rota and C. Yan \cite{KRY}). 
\begin{theorem} [Motzkin]
A totally non-negative matrix is variation-decreasing.
\end{theorem}
Apparently, once having known that the Lah triangular matrix ${\rm LM}_m$ is totally non-negative, we have that ${\rm LM}_m$ satisfies property (\ref{vardecr}).

\begin{corollary}
The Lah triangular matrix  ${\rm LM}_m$ is variation-decreasing.
\end{corollary}

\end{document}